\documentclass[11pt]{amsart}
\usepackage{amsopn}
\usepackage{amssymb, amscd}
\usepackage{multirow}
\usepackage{graphicx, graphics, epsfig}
\usepackage{faktor} 
\usepackage{enumerate}

\usepackage{tikz}
\usepackage{pgfplots}

\usepackage{hyperref}

\newcommand{\nc}{\newcommand}

\nc{\fg}{\mathfrak{f} } \nc{\vg}{\mathfrak{v} } \nc{\wg}{\mathfrak{w} }
\nc{\zg}{\mathfrak{z} } \nc{\ngo}{\mathfrak{n} } \nc{\kg}{\mathfrak{k} }
\nc{\mg}{\mathfrak{m} } \nc{\bg}{\mathfrak{b} } \nc{\ggo}{\mathfrak{g} }
\nc{\ggob}{\overline{\mathfrak{g}} } \nc{\sog}{\mathfrak{so} }
\nc{\sug}{\mathfrak{su} } \nc{\spg}{\mathfrak{sp} } \nc{\slg}{\mathfrak{sl} }
\nc{\glg}{\mathfrak{gl} } \nc{\cg}{\mathfrak{c} } \nc{\rg}{\mathfrak{r} }
\nc{\hg}{\mathfrak{h} } \nc{\tg}{\mathfrak{t} } \nc{\ug}{\mathfrak{u} }
\nc{\dg}{\mathfrak{d} } \nc{\ag}{\mathfrak{a} } \nc{\pg}{\mathfrak{p} }
\nc{\sg}{\mathfrak{s} } \nc{\affg}{\mathfrak{aff} } \nc{\qg}{\mathfrak{q} }

\nc{\pca}{\mathcal{P}} \nc{\nca}{\mathcal{N}} \nc{\lca}{\mathcal{L}}
\nc{\oca}{\mathcal{O}} \nc{\mca}{\mathcal{M}} \nc{\tca}{\mathcal{T}}
\nc{\aca}{\mathcal{A}} \nc{\cca}{\mathcal{C}} \nc{\gca}{\mathcal{G}}
\nc{\sca}{\mathcal{S}} \nc{\hca}{\mathcal{H}} \nc{\bca}{\mathcal{B}}
\nc{\dca}{\mathcal{D}} \nc{\val}{\operatorname{val}}

\nc{\vp}{\varphi} \nc{\ddt}{\frac{d}{dt}} \nc{\dds}{\frac{d}{ds}}
\nc{\dpar}{\frac{\partial}{\partial t}} \nc{\im}{\mathrm{i}}

\nc{\SO}{\mathrm{SO}} \nc{\Spe}{\mathrm{Sp}} \nc{\Sl}{\mathrm{SL}}
\nc{\SU}{\mathrm{SU}} \nc{\Or}{\mathrm{O}} \nc{\U}{\mathrm{U}} \nc{\Gl}{\mathrm{GL}}
\nc{\Se}{\mathrm{S}} \nc{\Cl}{\mathrm{Cl}} \nc{\Spein}{\mathrm{Spin}}
\nc{\Pin}{\mathrm{Pin}} \nc{\G}{\mathrm{GL}_n(\RR)} \nc{\g}{\mathfrak{gl}_n(\RR)}

\nc{\RR}{{\Bbb R}} \nc{\HH}{{\Bbb H}} \nc{\CC}{{\Bbb C}} \nc{\ZZ}{{\Bbb Z}}
\nc{\FF}{{\Bbb F}} \nc{\NN}{{\Bbb N}} \nc{\QQ}{{\Bbb Q}} \nc{\PP}{{\Bbb P}} \nc{\OO}{{\Bbb O}}

\nc{\vs}{\vspace{.2cm}} \nc{\vsp}{\vspace{1cm}} \nc{\ip}{\langle\cdot,\cdot\rangle}
\nc{\ipp}{(\cdot,\cdot)} \nc{\la}{\langle} \nc{\ra}{\rangle} \nc{\unm}{\tfrac{1}{2}}
\nc{\unc}{\tfrac{1}{4}} \nc{\und}{\tfrac{1}{16}} \nc{\no}{\vs\noindent}
\nc{\lam}{\Lambda^2(\RR^n)^*\otimes\RR^n} \nc{\tangz}{{\rm T}^{\rm Zar}}
\nc{\nor}{{\sf n}}  \nc{\mum}{/\!\!/} \nc{\kir}{/\!\!/\!\!/}
\nc{\Ri}{\tfrac{4\Ric_{\mu}}{||\mu||^2}} \nc{\ds}{\displaystyle}
\nc{\ben}{\begin{enumerate}} \nc{\een}{\end{enumerate}} \nc{\f}{\frac}
\nc{\lb}{[\cdot,\cdot]} \nc{\isn}{\tfrac{1}{||v||^2}}
\nc{\gkp}{(\ggo=\kg\oplus\pg,\ip)} \nc{\ukh}{(\ug=\kg\oplus\hg,\ip)}
\nc{\tgkp}{(\tilde{\ggo}=\kg\oplus\pg,\ip)}
\nc{\wt}{\widetilde} 
\nc{\iop}{\mathtt{i}} \nc{\jop}{\mathtt{j}}

\nc{\Hess}{\operatorname{Hess}} \nc{\ad}{\operatorname{ad}}
\nc{\Ad}{\operatorname{Ad}} \nc{\rank}{\operatorname{rk}}
\nc{\Irr}{\operatorname{Irr}} \nc{\End}{\operatorname{End}}
\nc{\Aut}{\operatorname{Aut}} \nc{\Inn}{\operatorname{Inn}}
\nc{\Der}{\operatorname{Der}} \nc{\Ker}{\operatorname{Ker}}
\nc{\Iso}{\operatorname{Iso}} \nc{\Diff}{\operatorname{Diff}}
\nc{\Lie}{\operatorname{L}} \nc{\tr}{\operatorname{tr}} \nc{\dif}{\operatorname{d}}
\nc{\sen}{\operatorname{sen}} \nc{\modu}{\operatorname{mod}}
\nc{\CRic}{\operatorname{PP}} \nc{\Cric}{\operatorname{P}} \nc{\Ricci}{\operatorname{Ric}}
\nc{\sym}{\operatorname{sym}} \nc{\herm}{\operatorname{herm}} \nc{\symac}{\operatorname{sym^{ac}}}
\nc{\symc}{\operatorname{sym^{c}}} \nc{\scalar}{\operatorname{scal}}
\nc{\grad}{\operatorname{grad}} \nc{\ricci}{\operatorname{Rc}}
\nc{\Nor}{\operatorname{Norm}}  \nc{\ricc}{\operatorname{Rc^{c}}}
\nc{\Ricc}{\operatorname{Ric^{c}}} \nc{\ricac}{\operatorname{Rc^{ac}}}
\nc{\Ricac}{\operatorname{Ric^{ac}}} \nc{\Riem}{\operatorname{Rm}} \nc{\Sec}{\operatorname{Sec}}
\nc{\riccig}{\operatorname{ric^{\gamma}}} \nc{\mm}{\operatorname{m}}
\nc{\Le}{\operatorname{L}} \nc{\tang}{\operatorname{T}}
\nc{\level}{\operatorname{level}} \nc{\rad}{\operatorname{r}}
\nc{\abel}{\operatorname{ab}} \nc{\CH}{\operatorname{CH}} \nc{\Cone}{{\mathcal C}} \nc{\CCone}{\operatorname{CC}} \nc{\CP}{{\mathcal P}}
\nc{\mcc}{\operatorname{mcc}} \nc{\Adj}{\operatorname{Adj}}
\nc{\Order}{\operatorname{O}}  \nc{\inj}{\operatorname{inj}} \nc{\proy}{\operatorname{pr}}
\nc{\vol}{\operatorname{vol}} \nc{\Diag}{\operatorname{Dg}} \nc{\Diagg}{\operatorname{Diag}}
\nc{\Spec}{\operatorname{Spec}} \nc{\Ima}{\operatorname{Im}} \nc{\Rea}{\operatorname{Re}}
\nc{\spann}{\operatorname{span}} \nc{\Aff}{\operatorname{Aff}}

\theoremstyle{plain}
\newtheorem{theorem}{Theorem}[section]
\newtheorem{proposition}[theorem]{Proposition}

\newtheorem{lemma}[theorem]{Lemma}

\theoremstyle{definition}
\newtheorem{definition}[theorem]{Definition}

\newtheorem{question}[theorem]{Question}
\newtheorem{conjecture}[theorem]{Conjecture}

\theoremstyle{remark}
\newtheorem{remark}[theorem]{Remark}

\newtheorem{example}[theorem]{Example}

\title{On Ricci negative Lie groups}

 \author{Jorge Lauret} \author{Cynthia E. Will}

\address{FaMAF, Universidad Nacional de C\'ordoba and CIEM, CONICET (Argentina)}
\email{lauret@famaf.unc.edu.ar} \email{cwill@famaf.unc.edu.ar}

\thanks{This research was partially supported by grants from FONCyT and SeCyT (UNC)}

\begin{document}

\maketitle

\begin{abstract}
We give an overview of what is known on Lie groups admitting a left-invariant metric of negative Ricci curvature, including many natural questions and conjectures in the solvable case.  We also introduce an open and convex cone $\cca(\ngo)$ of derivations attached to each nilpotent Lie algebra $\ngo$, which is defined as the image of certain moment map and parametrizes a set of solvable Lie algebras with nilradical $\ngo$ admitting Ricci negative metrics.  
\end{abstract}

\tableofcontents

\section{Introduction}\label{intro}

The main objects of study in this paper can be introduced as generalizations of the Euclidean space as follows: Riemannian manifolds which are also groups and left translations are all isometries.  This somewhat naive presentation of Lie groups endowed with left-invariant metrics could be hiding the rich variety of possible curvature behaviors this class of metrics exhibits, as shown by Milnor in his seminal paper \cite{Mln}.    

Given a Lie group $G$, it is natural to expect a nice and strong interplay between any prescribed curvature condition of left-invariant Riemannian metrics on $G$ and not only the topology of $G$ but also the algebraic structure of its Lie algebra $\ggo$.  Structurally, the only pinching curvature behavior which is still not understood in the homogeneous case is $\Ricci<0$ (see e.g.\ the overview given in \cite[Introduction]{NklNkn}).  Indeed, the following natural problem is wide open and represents our leitmotiv in this paper.  

\begin{question}\label{MQ}
Which Lie groups admit a left-invariant metric of negative Ricci curvature?  
\end{question}

We recall that in the general case there are no topological obstructions on a differentiable manifold $M$ to the existence of a complete Riemannian metric with $\Ricci<0$ (see \cite{Lhk}).

After surveying the subject in Sections \ref{his-sec} and \ref{deg-sec}, we propose several natural questions and conjectures in Section \ref{RN-sec} in the context of solvable Lie groups.  Finally, some new results are proved in Section \ref{cone-sec}, including the construction of an open and convex cone $\cca(\ngo)$, attached to each nilpotent Lie algebra $\ngo$, parametrizing a subset (conjectured to be essentially the full set) of one-dimensional solvable extensions of $\ngo$ admitting $\Ricci<0$ metrics.

\vs \noindent {\it Acknowledgements.} The authors are very grateful to Romina Arroyo and Emilio Lauret for very helpful comments.

\section{Some history}\label{his-sec}

In what follows, we overview the rich history of the existence problem for negatively curved metrics on Lie groups (from now on, any metric on a Lie group is assumed to be left-invariant).

\subsection{Negative sectional} 
We start with four theorems concerning scalar and sectional curvature.  

\begin{theorem}\label{scal}\cite{Mln} 
Any non-abelian Lie group admits $\scalar<0$ metrics.  Moreover, any non-flat metric on a solvable Lie group has $\scalar<0$.   
\end{theorem}

\begin{theorem}\label{sec}\cite{Hnt,AznWls}  
A Lie group admits $\Sec<0$ metrics if and only if, 
\begin{enumerate}[{\rm (i)}]
\item $\dim{\ngo}=\dim{\sg}-1$, where $\sg$ denotes the corresponding Lie algebra and $\ngo$ is the {\it nilradical} of $\sg$ (i.e.\ the maximal nilpotent ideal); in particular, $\sg$ is solvable.

\item There is an element $Y\in\sg$ such that $\ad{Y}|_{\ngo}>0$ (i.e.\ its eigenvalues have all positive real parts).  
\end{enumerate}
\end{theorem}

Thus the structure of a $\Sec<0$ Lie algebra is simply given by a nilpotent Lie algebra together with a positive derivation.   
Unfortunately, the question of which nilpotent Lie algebras admit a positive derivation seems to be hopeless.

\begin{theorem}\label{secp}\cite{EbrHbr} 
A (solvable) Lie group admits a metric such that $-4\leq\Sec\leq-1$ if and only if,
\begin{enumerate}[{\rm (i)}]
\item $\ngo$ is $2$-step nilpotent or abelian.  

\item There is an element $Y\in\sg$ such that the real part of any eigenvalue of $\ad{Y}|_{\ngo}$ belongs to the interval $[1,2]$.   

\item $\ad{Y}|_{\ngo}$ is given by $\alpha I$ when restricted to the generalized eigenspace of any eigenvalue $\alpha$ such that $\Rea(\alpha)=1$ or $2$.   
\end{enumerate}  
\end{theorem}

It is further obtained in \cite{EbrHbr} that if $\ngo$ is non-abelian, then $\ad{Y}|_{\ngo}$ must actually be, up to amalgamated products, the canonical derivation of a $2$-step algebra given by $2I$ on the center and $I$ on its orthogonal complement.  

\begin{theorem}\label{sec0}\cite{AznWls,Wlf,Alk}  
A Lie group admits a $\Sec\leq 0$ metric if and only if 
\begin{enumerate}[{\rm (i)}]
\item $[\ag,\ag]=0$, where $\sg=\ag\oplus\ngo$ is the orthogonal decomposition; in particular, $\sg$ is solvable.

\item There is an element $Y\in\sg$ such that $\ad{Y}|_{\ngo}\geq 0$. 

\item For any $Y'\in\ag$, $\ad{Y'}|_{\ngo}$ coincides with $\alpha I$ when restricted to the generalized eigenspace of any eigenvalue $\alpha$ such that $\Rea(\alpha)=0$.  

\item Some extra technical conditions.  
\end{enumerate}
Moreover, any homogeneous Riemannian manifold with $\Sec\leq 0$ is isometric to a metric on a simply connected solvable Lie group.  
\end{theorem}

\subsection{Negative Ricci, non-solvable case}
Let us now consider the curvature behavior addressed in Question \ref{MQ}: $\Ricci<0$.  Unlike negative sectional curvature, non-solvable Lie groups come into play in the Ricci negative case. 

\begin{theorem}\label{unims}\cite{Dtt}  
No unimodular solvable Lie group can admit $\Ricci<0$ metrics. 
\end{theorem}

\begin{theorem}\label{unim}\cite{DttLtMtl}  
The only unimodular Lie groups that can admit $\Ricci<0$ metrics are the non-compact semisimple ones. 
\end{theorem}

The following is a list of some of the known results, constructions and examples in the non-solvable case.  

\begin{enumerate}[{\small $\bullet$}]
\item \cite{DttLt}  $\Sl_n(\RR)$ admits $\Ricci<0$ metrics for any $n\geq 3$.  The corresponding compact quotients $\Sl_n(\RR)/\Gamma$ by any lattice $\Gamma$ gave the first examples of compact manifolds admitting both a $\scalar>0$ metric and a $\Ricci<0$ metric, answering a question in \cite[pp. 242]{Yau}.   
\item[ ]
\item \cite{DttLtMtl}  Most non-compact simple Lie groups do admit $\Ricci<0$ metrics, with some low dimensional exceptions, including $\Sl_2(\CC)$, $\Spe(2,\RR)$ and $G_2^*$.  The existence of $\Ricci<0$ metrics on these simple Lie groups is still open. 
\item[ ]
\item \cite{Mln} $\Sl_2(\RR)$ does not admit $\Ricci<0$ metrics.  This is actually the only simple Lie group for which the non-existence of $\Ricci<0$ metrics has been proved.  Moreover, we do not know of any example of a non-unimodular non-solvable Lie group $G$ such that $G$ admits no $\Ricci<0$ metric.
\item[ ]
\item \cite{JblPtr}  A semisimple Lie group admitting a metric with $\Ricci<0$ can not have compact factors, i.e.\ it is of non-compact type.  On the other hand, we do not know whether the existence of $\Ricci<0$ metrics on a semisimple Lie group implies the existence of such metrics on each of its factors.  
\item[ ]
\item \cite{Wll1, Wll2} Unexpected examples of Lie groups admitting $\Ricci<0$ metrics which are neither semisimple nor solvable.  Moreover, the Levi factors of some of these examples are compact, including $\SU(n)$ ($n\geq 2$), $\SO(n)$ ($n\geq 3$) and $\Spe(n)$ ($n\geq 2$).  
\item[ ]
\item \cite{Wll2}  Any non-compact semisimple Lie group admitting a $\Ricci<0$ metric can be the Levi factor of a non-semisimple Lie group with a $\Ricci<0$ metric.  Non-abelian nilradicals are possible in most of these  constructions.  
\item[ ]
\item \cite{Wll1} For any (non-trivial) real representation $W$ of $\sug(2)$, the Lie algebra 
$$
\ggo=(\sug(2)\oplus\RR)\ltimes W
$$ 
admits $\Ricci<0$ metrics, where $\RR$ is acting on $W$ by possibly different multiples of the identity on each irreducible component.  This produces the example of lowest possible dimension of a $\Ricci<0$ non-semisimple and non-solvable Lie algebra: $\ggo=(\sug(2)\oplus\RR)\ltimes\RR^3$, $\dim{\ggo}=7$. 
\item[ ]
\item \cite{LrtWll}  Any compact semisimple Lie group can be the Levi factor of a Lie group admitting a $\Ricci<0$ metric.  Moreover, given a compact simple Lie algebra $\ug$, all but finitely many irreducible representations of its complexification $\ug\otimes\CC$ (including the adjoint representation) can be the abelian nilradical of a Lie algebra $\ggo$ with Levi factor $\ug$ admitting $\Ricci<0$ metrics.  We have listed in Table \ref{dims}, for each $\ug$, the example of lowest dimension known of such a $\ggo=(\ug\oplus\RR)\ltimes V$, where $(V,\pi)$ is the (complex) irreducible representation of $\ug\otimes\CC$ viewed as a real representation of $\ug$; in particular 
$$
\dim{\ggo}=\dim{\ug}+1+2\dim_\CC{V}.
$$   
In Table \ref{dims}, the fundamental representations of  $\ug\otimes\CC$ are denoted by $\omega_1,\omega_2,\dots$ and so on in the order used by the software package LiE.  
\end{enumerate}

{\small 
\begin{table}
\caption{Smallest dimension known of a $\Ricci<0$ Lie algebra $\ggo=(\ug\oplus\RR)\ltimes V$ with a given compact simple Levi factor $\ug$.} \label{dims}
\begin{tabular}{|c|c|c|c|c|}
\hline 
&&&& \\ 
\text{Type} & $\mathfrak{u}$ & $\dim{\ug}$ & $\pi$ & $\dim{\ggo}$  \\
&&&& \\ \hline \hline 
&&&& \\
$\textup{A}_n,\, n\geq 2$ & $\mathfrak{su}(n+1)$ & $n(n+2)$ & $2\omega_1=Sym^2(\CC^{n+1})$ & $2n^2+5n+3$
\\ &&&& \\ 
\hline
&&&& \\ 
$\textup{B}_n,\,n \ge 2$ &$ \mathfrak{so}(2n+1)$ & $n(2n+1)$ & $  \omega_2=\ad$ &$  6n^2+3n+1 $ 
\\ &&&& 
\\ 
\hline
&&&& \\ 
$\textup{C}_n,\,n\geq 3 $ &$ \mathfrak{sp}(n) $ & $n(2n+1)$ & $  2\omega_1 = \ad $ &$ 6n^2+3n+1 $ 
\\ &&&& \\
\hline
&&&& \\ 
$\textup{D}_n,\,n \ge 4$ &$ \mathfrak{so}(2n) $ & $n(2n-1)$ & $  \omega_2=\ad $ &$ 6n^2-3n+1$ 
\\ &&&& \\
\hline
&&&& \\ 
$\textup{E}_{6} $ &$ E_6 $ &  $78$ &$\omega_2=\ad $  & $235$ 
\\ &&&& 
\\
\hline
&&&& \\ 
$\textup{E}_{7} $ &$ E_7 $ & $133$ & $\omega_1=\ad $ & $400$ 
\\ &&&& 
\\
\hline
&&&& \\ 
$\textup{E}_{8} $ &$E_8  $ & $248$ & $\omega_8=\ad $ & $745$ 
\\ &&&& 
\\
\hline
&&&& \\ 
$\textup{F}_{4} $ &$F_4$ & $52$ & $\omega_1=\ad $ & $157$  
\\ &&&& 
\\
\hline
&&&& \\ 
$\textup{G}_{2} $ &$ G_2  $ & $14$ & $ \omega_2 = \ad $ & 43 
\\ &&&& \\ \hline 
\end{tabular}
\end{table}
}

\subsection{Negative Ricci, solvable case}
We now focus on the existence of Ricci negative metrics in the solvable case.  Note that unimodular solvable (in particular nilpotent) Lie groups are not allowed in the Ricci negative case by Theorem \ref{unims}.  Recall the orthogonal decomposition $\sg=\ag\oplus\ngo$, where $\ngo$ is the nilradical of the solvable Lie algebra $\sg$.  

\begin{enumerate}[{\small $\bullet$}]
\item \cite{NklNkn}  {\it Sufficient condition}: There exists $Y\in\sg$ such that $\ad{Y}|_\ngo>0$ (recall that this means $\Rea\Spec(\ad{Y}|_\ngo)>0$).  Note that the nilradicals involved in these examples are the same as those needed for $\Sec<0$ (cf.\ Theorem \ref{sec}), although the condition $[\ag,\ag]=0$ is not mandatory here as in the case of $\Sec\leq 0$ (cf.\ Theorem \ref{sec0}).  
\item[ ]
\item \cite{NklNkn} {\it Necessary condition}: There exists $Y\in\sg$ such that $\tr{\ad{Y}}>0$ and $\ad{Y}|_{\zg(\ngo)}>0$, where $\zg(\ngo)$ is the center of $\ngo$.  Surprisingly (or not), this is the only known general obstruction so far.  
\end{enumerate}

For a given $D\in\Der(\ngo)$, we consider its additive Jordan decomposition,
$$
D=D^{\RR}+D^{\im\RR}+D^n,  \qquad [D^{\RR},D^{\im\RR}]=0, \quad [D^\RR,D^n]=0, \quad [D^{\im\RR},D^n]=0,
$$
where $D^{\RR}$ is diagonalizable (over $\RR$), $D^{\im\RR}$ is diagonalizable over $\CC$ with only imaginary eigenvalues and $D^n$ is nilpotent.  It is well known that $D^{\RR},D^{\im\RR},D^n\in\Der(\ngo)$.  A maximal abelian subspace $\tg(\ngo)$ of diagonalizable (over $\RR$) derivations is called a {\it maximal torus} of $\ngo$, which is known to be unique up to $\Aut(\ngo)$-conjugation.  

In the light of the above and some other results in the particular cases of Heisenberg and filiform Lie algebras as nilradicals obtained in \cite{NklNkn,Nkl}, a complete characterization of solvable Lie algebras admitting $\Ricci<0$ metrics is expected to take the following form.

\begin{conjecture}\label{Ys}\cite{NklNkn}
For each nilpotent Lie algebra $\ngo$, there is an open and convex cone $\cca\subset\tg(\ngo)$ such that a solvable Lie algebra $\sg$ with nilradical $\ngo$ admits a $\Ricci<0$ metric if and only if there exists $Y\in\sg$ such that $\ad{Y}|_{\ngo}^\RR\in\cca$ (up to automorphism conjugation).  
\end{conjecture}

The following rank-one reduction was proved to hold in the Einstein case in \cite{Hbr}.  

\begin{conjecture}\label{rank-one}\cite{NklNkn}
A solvable Lie algebra $\sg$ with nilradical $\ngo$ admits a $\Ricci<0$ metric if and only if there exists $Y\in\sg$ such that the Lie subalgebra $\RR Y\oplus\ngo$ admits a $\Ricci<0$ metric.    
\end{conjecture}

On the other hand, the structural conditions on a solvable Lie algebra imposed by the negativity of sectional or Ricci curvature described above motivates the following natural question.  

\begin{question}\label{which}
Which nilpotent Lie algebras can be the nilradical of a solvable Lie algebra admitting $\Ricci<0$ metrics?
\end{question}

Such a Lie algebra will be called a {\it Ricci negative nilradical} (RN-nilradical for short).  Since the existence of a positive derivation is sufficient, even for the existence of a $\Sec<0$ solvable extension by Theorem \ref{sec}, any nilpotent Lie algebra which is $2$-step or has dimension $\leq 6$ is a RN-nilradical.  It is also known that any $7$-dimensional nilpotent Lie algebra (admitting a non-nilpotent derivation, or equivalently, a non-nilpotent solvable extension) is a RN-nilradical (see \cite{RN}).  However, the following examples suggest that, as in the case of Einstein nilradicals (see Remark \ref{E-rem} below), a structural characterization of RN-nilradicals may be hopeless. 

\begin{enumerate}[{\small $\bullet$}]
\item \cite{RN}  Explicit examples of nilpotent Lie algebras with the following properties (cf.\ the necessary condition above):  
\begin{enumerate}[(i)]
\item There are no diagonalizable derivations but not every derivation is nilpotent.  

\item There is a diagonalizable derivation but any derivation is traceless.  
\end{enumerate}
\item[ ]
\item \cite{RN}  Explicit examples of nilpotent Lie algebras which are not RN-nilradicals.  They all have a derivation of positive trace and one example for each of the following features is provided (for the first three examples any diagonalizable derivation has a zero eigenvalue on the center):
\begin{enumerate}[(i)]
  \item $\dim{\ngo}=8$ ($5$-step).

  \item $\ngo$ is $3$-step nilpotent ($\dim{\ngo}=10$).

  \item A continuous family of pairwise non-isomorphic $6$-step nilpotent Lie algebras of dimension $13$.

  \item $\ngo$ has a non-singular derivation but any diagonalizable derivation has a negative eigenvalue on the center ($\dim{\ngo}=13,17$, $5$-step).
\end{enumerate}
\item[ ]
\item \cite{RN}  A $10$-dimensional $5$-step example of a RN-nilradical such that any diagonalizable derivation has a negative eigenvalue.    
\end{enumerate}

\begin{remark}\label{E-rem}
If in Question \ref{which} we replace $\Ricci<0$ by {\it Einstein} (i.e.\ $\Ricci_g=cg$), then what it is analogously obtained are the so called {\it Einstein nilradicals} (see the survey \cite{cruzchica} for further information).  To admit a derivation with natural numbers as eigenvalues (so defining an $\NN$-gradation) is the only known general obstruction for Einstein nilradicals and no sufficient structural condition is available.  Several classification results within special classes of algebras are known in the literature though.  Remarkably, Einstein nilradicals are precisely the nilpotent Lie algebras admitting a Ricci soliton metric, which is known to be unique up to isometry and scaling.      
\end{remark}

\begin{remark}
There are also in the literature fine results on the existence of negatively curved $G$-invariant metrics on a homogeneous space $G/K$.  For instance, the homogeneous space $\SO^+(n,2)/\SO(n)$ ($n\geq 2$), which is homeomorphic to $S^1\times\RR^k$, does admit a $\Ricci<0$ invariant metric (see \cite[Example 1]{Nkn}).  On the other hand, it is proved in \cite{Brr} that any invariant metric on a given $G/K$ has $\scalar<0$ if and only if the universal cover of $G/K$ is the Euclidean space.  
\end{remark}

\section{Degeneration Principle}\label{deg-sec}

It is evident that all the geometric information on a Lie group endowed with a left-invariant metric, say $(G,g)$, is encoded just in the inner product $\ip:=g_e$ and the Lie bracket of the Lie algebra $\ggo$.  This suggests the following point of view, called the {\it moving-bracket approach}, which proposes to vary Lie brackets rather than inner products.  

We fix a real vector space $\ggo$ endowed with a fixed inner product $\ip$.  Let $\lca\subset\Lambda^2\ggo^*\otimes\ggo$ denote the algebraic subset of all Lie brackets on $\ggo$.  Each $\mu\in\lca$ is identified with the (left-invariant) metric determined by $\ip$ on the simply connected Lie group $G_\mu$ with Lie algebra $(\ggo,\mu)$:
$$
\mu \longleftrightarrow (G_\mu,\ip).
$$
In this way, the isomorphism class $\Gl(\ggo)\cdot\mu$ is identified with the set of all metrics on $G_\mu$ as follows:
$$
(G_{h\cdot\mu},\ip) \longleftrightarrow (G_\mu,\la h\cdot,h\cdot\ra), \qquad\forall h\in\Gl(\ggo),
$$
where $h\cdot\mu:=h\mu(h^{-1}\cdot,h^{-1}\cdot)$.  Note that $h^{-1}$ is an isometric isomorphism that determines an isometry between these two Riemannian manifolds.  Thus any two Lie brackets in the same $\Or(\ggo,\ip)$-orbit are isometric as Riemannian metrics; the converse assertion is known to hold (see \cite{Alk}) for {\it completely solvable} Lie brackets (i.e.\ $\Spec(\ad_\mu{X})\subset\RR$ for any $X\in\ggo$).  

A key observation is that any kind of geometric quantity associated to the Riemannian manifold $(G_\mu,\ip)$ depends continuously on the Lie bracket $\mu\in\lca$, which can be used to study pinching curvature properties as follows.

\begin{definition}\label{degen-def}
Given $\mu,\lambda\in\lca$, we say that $\mu$ {\it degenerates to} $\lambda$, denoted by $\mu\rightarrow\lambda$, if $\lambda\in\overline{\Gl(\ggo)\cdot\mu}$, where the closure relative to the usual vector space topology of $\Lambda^2\ggo^*\otimes\ggo$ is considered.
\end{definition}

Recall the geometric role of the orbit $\Gl(\ggo)\cdot\mu$ as the set of all metrics on $G_\mu$.

\begin{proposition}\label{deg-pin} (Degeneration Principle, see Figure \ref{degP})
If $\mu\rightarrow\lambda$ and $G_\lambda$ admits a metric satisfying a strict pinching curvature condition (e.g.\ $\Ricci<0$), then there is also a metric on $G_\mu$ for which the same pinching curvature condition holds.
\end{proposition}

\begin{proof}
Without any lost of generality, we can assume that $(G_\lambda,\ip)$ has $\Ricci<0$.  If $\mu\rightarrow\lambda$, then there is a sequence $h_k\in\Gl(\ggo)$ such that $h_k\cdot\mu\to\lambda$, as $k\to\infty$.  The pinching curvature condition will therefore hold for $(G_{h_k\cdot\mu},\ip)$ for sufficiently large $k$ by continuity, which implies that it holds for the corresponding isometric metrics $\la h_k\cdot,h_k\cdot\ra$ on $G_\mu$ for sufficiently large $k$, concluding the proof.
\end{proof}

\begin{figure}
  \begin{tikzpicture}[scale=0.8, domain=-4:3]
    \draw[very thin,color=gray] (-2.5,-1.5) grid (4.5,6.5);
    \draw[very thick, color=blue, fill=orange, opacity=0.8,
    rotate=-30]  plot (\x,{0.3*\x*\x})  ;
    \draw[very thick, color=blue, fill=blue] (0,0) circle (1.5pt) node[below] {$\lambda$};
    \draw (2.3,4) node[thick, color=orange] {$GL(\ggo)\cdot\mu$};
     \draw (3,0) node[below, color=blue] {$GL(\ggo)\cdot\lambda$};
    \draw[very thick, fill] (2,1) circle (1pt);
    \draw[very thick, fill] (1.5,0.25*1.5*1.5) circle (1pt) node[right] {$h_k\cdot\mu$};
    \draw[very thick, fill] (1,0.25*1) circle (1pt);
    \draw[very thick, fill] (0.6,0.25*0.6*0.6) circle (1pt);
    \draw[very thick, fill] (0.2,0.25*0.2*0.2) circle (1pt);
    \draw[fill, opacity=0.2] (0,0) circle [radius=0.8];
    \draw (0,-1) node  {$\Ricci<0$};
    \end{tikzpicture}
  \caption{Degeneration Principle}\label{degP}
\end{figure}
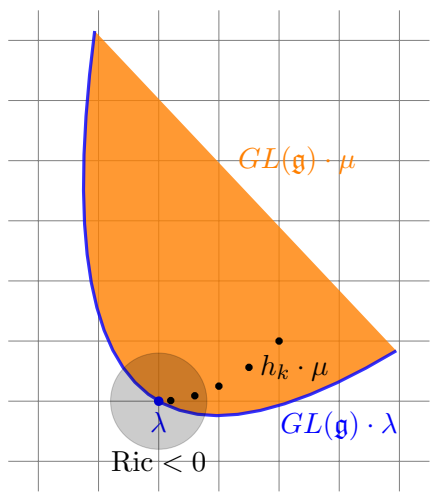

We now overview many applications, starting back in the 70s, of the moving-bracket approach and the above principle in the study of negatively curved Lie groups.  

\begin{enumerate}[{\small $\bullet$}]
\item Milnor proved the first part of Theorem \ref{scal} by showing that any non-abelian Lie algebra of dimension $n$ degenerates to exactly one of the following: 
$$
\mu_{heis}(e_1,e_2)=e_3; \qquad \mu_{hyp}(e_n,e_i)=e_i, \quad i=1,\dots,n-1,
$$
which are easily seen to admit $\scalar<0$ metrics.  These two are actually the only non-abelian Lie groups admitting a unique metric up to isometry and scaling (see \cite{inter}), and the metric on $\mu_{hyp}$ is isometric to the real hyperbolic space $\RR H^n$.  
\item[ ] 
\item To show that conditions (i) and (ii) of Theorem \ref{sec} are sufficient, Heintze considered the curve of Lie brackets $\mu_t$ on $\sg$ defined by $\ad_{\mu_t}{Y}|_\ngo=t\ad{Y}|_{\ngo}$, $\mu_t|_{\ngo\times\ngo}=\lb_\ngo$ and proved that $\Sec_{\mu_t}<0$ for sufficiently large $t$ (see \cite[Theorem 2]{Hnt}).   
\item[ ] 
\item In order to prove the above mentioned sufficient condition, Nikolayevsky-Nikonorov first showed that any solvable Lie algebra $\mu\in\lca$ degenerates to a $\lambda$ such that $\lambda(\ag,\ag)=0$, $\lambda(\ngo,\ngo)=0$, $\lambda(\ag,\ngo)\subset\ngo$ and $\ad_\lambda{X}$ is a semisimple operator with the same spectrum as $\ad_\mu{X}$ for any $X\in\ag$, where $\ggo=\ag\oplus\ngo$ is any decomposition such that $\ngo$ is the nilradical of $\mu$.  Then they also applied the Degeneration Principle to prove the sufficient conditions in \cite[Theorems 3 and 4]{NklNkn} and Nikolayevsky did it in the proof of \cite[Theorem 1.3]{Nkl}.  
\item[ ] 
\item A main part of most constructions of a non-solvable Lie algebra $\ggo$ admitting $\Ricci<0$ metrics in \cite{Wll1, Wll2} is to show that $\ggo$ degenerates to a solvable Lie algebra $\sg=\ag\oplus\ngo$ such that $[\ag,\ag]=0$.  Interestingly, $\ngo$ is non-abelian in spite of the nilradical of $\ggo$ is abelian.   
\item[ ] 
\item In \cite[Theorem 3.2]{RN}, the Degeneration Principle is applied to obtain a characterization of Ricci negative one-dimensional solvable extensions of a given nilpotent Lie algebra in terms of the set of all Ricci operators at the points in the closure $X$ of an orbit of nilpotent Lie brackets.  Since $X$ is irreducible (relative to the real algebraic Zariski topology), this allowed the application of nice convexity properties of a certain moment map (see Theorem \ref{main} below).   
\end{enumerate}

\begin{remark}
The moving-bracket approach was considered by J. Heber in his seminal article \cite{Hbr} on Einstein solvmanifolds and it has been further developed in \cite{soliton,minimal, spacehm}.  We refer to \cite[Section 5]{BF} and \cite{BhmLfn2,BhmLfn3,sol-HS} for more applications.  In most of the applications, concepts and results from geometric invariant theory, including moment maps, closed orbits, stability, categorical quotients and Kirwan stratification, have been exploited in one way or another.
\end{remark}

\section{Ricci negative derivations}\label{RN-sec}

As explained in Section \ref{his-sec}, the diverse kinds of examples provided in \cite{RN} seems to indicate that Question \ref{which} is too ambitious.  In this section, we consider instead the following also very natural problem.  

\begin{question}\label{fixnil}
Given a nilpotent Lie algebra $\ngo$, which are the solvable Lie algebras with nilradical $\ngo$ admitting a $\Ricci<0$ metric?
\end{question}

With the aim of keeping it simple, we first focus on one-dimensional solvable extensions.  Let $\ngo$ be a nilpotent Lie algebra.  Each $D\in\Der(\ngo)$ defines a solvable Lie algebra
$$
\sg_D=\RR f\oplus\ngo,
$$
with Lie bracket defined as the semi-direct product such that $\ad{f}|_{\ngo}=D$.

\begin{definition}\label{rnd-def}
A derivation $D$ of a nilpotent Lie algebra $\ngo$ with $\tr{D}>0$ is said to be {\it Ricci negative} (or rn-derivation for short) if the solvable Lie algebra $\sg_D$ admits an inner product of negative Ricci curvature.  We denote by $\Der(\ngo)_{rn}$ the set of all Ricci negative derivations of $\ngo$.   
\end{definition}

It follows from Theorem \ref{sec} that any positive derivation $D$ (i.e. $\Spec\left(D^\RR\right)>0$) belongs to $\Der(\ngo)_{rn}$.  Note also that any $\ngo$ admitting a rn-derivation is a RN-nilradical; the validity of the converse assertion, which is strongly related to Conjecture \ref{rank-one}, is not known.  

\begin{question}\label{RNrn}
Does any RN-nilradical admit a rn-derivation? 
\end{question}

The following particular instance of Question \ref{fixnil} arises.   
 
\begin{question}\label{Drn}
What kind of set is $\Der(\ngo)_{rn}$? 
\end{question}

Note that $\Der(\ngo)_{rn}$ parametrizes the set of all one-dimensional solvable extensions of $\ngo$ admitting $\Ricci<0$ metrics.

\subsection{Ricci operator}\label{Ric-sec}
Before going through more questions and conjectures, it is worth pointing out at this point some of the technical difficulties involved.    

We fix an inner product  $\ip$ on $\sg=\RR f\oplus\ngo$ such that $|f|=1$ and $f\perp\ngo$, in order to consider the moving-bracket approach described in Section \ref{deg-sec}.  Let $\lb_\ngo$ denote the Lie bracket of $\ngo$.  The Lie bracket of $\sg_D$ is therefore determined by the pair $(D,\lb_\ngo)$, for any $D\in\Der(\ngo)$.  It is easy to see (using e.g. \cite{solvsolitons} or \cite[(11)]{alek}) that the Ricci operator of $(\sg_D,\ip)$ is given by
\begin{equation}\label{Ric-half}
\Ricci_{(D,\lb_\ngo)}=\left[\begin{array}{c|c}
-\tr{S(D)^2} & \ast \\\hline
& \\
\ast & \Ricci_{\lb_\ngo}+\unm[D, D^t]-\tr(D)S(D) \\ &
\end{array}\right],
\end{equation}
where $\Ricci_{\lb_\ngo}=\tfrac{|\lb_\ngo|^2}{4}\mm(\lb_\ngo)$ is the Ricci operator of $(\ngo,\ip|_{\ngo\times\ngo})$ (see \eqref{defmm-2} for the definition of the moment map $\mm$),  $S(D):=\unm(D+D^t)$ and
$$
\la\Ricci f,X\ra=-\tr{S(D)\ad_\ngo{X}}, \qquad\forall X\in\ngo.
$$
It is known that $\ast=0$ if $D$ is normal (see e.g.\ the proof of \cite[Proposition 4.3]{solvsolitons}).

If with respect to a fixed orthonormal basis $\{ f, e_1,\dots,e_n\}$ of $(\sg,\ip)$, where $e_i\in\ngo$ for all $i$, 
\begin{equation}\label{hbarra}
\overline{h}=\left[\begin{array}{c|ccc}
c^{-1} & &0& \\ \hline &\\ 
X&&h& \\&
\end{array}\right], \qquad c\in\RR^*, \quad X\in\ngo,  \quad h\in\Gl(\ngo), 
\end{equation}
then $\left(\sg_D,\la\overline{h}\cdot,\overline{h}\cdot\ra\right)$ is isometric to the pair 
\begin{equation}\label{cXh}
\left(ch(D-\ad_\ngo{h^{-1}X})h^{-1},h\cdot\lb_\ngo\right), 
\end{equation}
since $\overline{h}$ is an isometric isomorphism between the corresponding metric Lie algebras.   The fact that any metric on $\sg_D$ is of the form $\la\overline{h}\cdot,\overline{h}\cdot\ra$ for some $\overline{h}$ as above implies the following characterization of Ricci negative derivations of $\ngo$.  

\begin{lemma}\label{rn-cond}
$D\in\Der(\ngo)_{rn}$ if and only if there exist  $X\in\ngo$ and $h\in\Gl(\ngo)$ such that $\Ricci_{(D_{X,h},h\cdot\lb_\ngo)}<0$, where $D_{X,h}:=(h(A-\ad_\ngo{h^{-1}X})h^{-1}\in\Der(\ngo,h\cdot\lb_{\ngo})$.    
\end{lemma} 

The terms $[D_{X,h},(D_{X,h})^t]$ and $S(D_{X,h})$ in formula \eqref{Ric-half} for $\Ricci_{(D_{X,h},h\cdot\lb_\ngo)}$ are extremely difficult to control, even in the case when $D$ is diagonalizable.

\subsection{Diagonalizable Ricci negative derivations} 
We now list some immediate though very useful consequences of Definition \ref{rnd-def}.  

\begin{enumerate}[{\small $\bullet$}]
\item $\Der(\ngo)_{rn}$ is open in $\Der(\ngo)$ and it is a {\it cone} (i.e.\ $\RR_{>0}$-invariant).  
\item[ ]
\item $\Der(\ngo)_{rn}$ is $\Aut(\ngo)$-invariant.  Moreover, if $D\in\Der(\ngo)_{rn}$, then 
$$
D':=c\left(hDh^{-1}-\ad_\ngo{X}\right)\in\Der(\ngo)_{rn},
$$ 
for any $c>0$, $X\in\ngo$ and $h\in\Aut(\ngo)$ (see \eqref{cXh}).  Indeed, it is easy to check that $\sg_{D'}$ is isomorphic to $\sg_D$.    
\item[ ]
\item $\Der(\ngo)_{rn}^\RR:=\{ D^\RR:D\in\Der(\ngo)_{rn}\}$ is an open cone in the $\Aut(\ngo)$-invariant subspace $\Der(\ngo)^\RR:=\{ D^\RR:D\in\Der(\ngo)\}$.  We do not known a priori if $D^\RR$ is also Ricci negative for any $D\in\Der(\ngo)_{rn}$.  
\item[ ]
\item $\Der(\ngo)_{rn}^\RR=\Aut(\ngo)\cdot\tg(\ngo)_{rn}$, where 
\begin{equation}\label{trn-def}
\tg(\ngo)_{rn}:=\Der(\ngo)_{rn}^\RR\cap\tg(\ngo),
\end{equation}
and $\tg(\ngo)$ is any maximal torus of diagonalizable derivations of $\ngo$ (cf.\ Conjecture \ref{Ys}).  Note that $\Der(\ngo)^\RR=\Aut(\ngo)\cdot\tg(\ngo)$ since all the maximal tori are $\Aut(\ngo)$-conjugate and any diagonalizable derivation belongs to some maximal torus.   
\item[ ]
\item $\tg(\ngo)_{rn}$ is an open cone in $\tg(\ngo)$. 
\item[ ]
\item The {\it Weyl group} 
\begin{equation}\label{Wg}
W(\ngo):=N_{\Aut(\ngo)}(\tg(\ngo))/C_{\Aut(\ngo)}(\tg(\ngo))
\end{equation}
acts on $\tg(\ngo)_{rn}$ by conjugation, where $N$ and $C$ denote normalizer and centralizer, respectively.  
\end{enumerate}

It is not too hard to see that if Conjecture \ref{Ys} turns out to be true, then the following must hold.  

\begin{conjecture}\label{Diff}
$D\in\Der(\ngo)_{rn}$ if and only if $D^\RR\in\Der(\ngo)_{rn}$.     
\end{conjecture}

This would imply that $\tg(\ngo)_{rn}=\Der(\ngo)_{rn}\cap\tg(\ngo)$ and so $\tg(\ngo)_{rn}$ stands out from the sets to be understood in order to give a satisfactory answer to Question \ref{Drn}.  The relationship with the image of certain moment map and its convexity properties described below lead us to propose the following conjecture.   

\begin{conjecture}\label{trn}
The open cone $\tg(\ngo)_{rn}$ is convex.  
\end{conjecture}

It is easy to see that $\tg(\ngo)_{rn}$ must contain the cone $\cca$ in Conjecture \ref{Ys}, provided Conjecture \ref{rank-one} holds true.

\subsection{Strongly Ricci negative derivations and the moment map}\label{srn-sec}
The following stronger condition on a derivation was studied in \cite{RN}.  

\begin{definition}\label{srnd-def}
A derivation $D$ of a nilpotent Lie algebra $\ngo$ with $\tr{D}>0$ is said to be {\it strongly Ricci negative} (srn-derivation for short) if the solvable Lie algebra $\sg_D$ admits an inner product of negative Ricci curvature such that $D^t=D$ and $f\perp\ngo$.  We denote by $\Der(\ngo)_{srn}$ the cone of all strongly Ricci negative derivations of $\ngo$.   
\end{definition}

Note that $\Der(\ngo)_{srn}\subset\Der(\ngo)^\RR$ and it is also $\Aut(\ngo)$-invariant.  Indeed, if $(\sg_D,\ip)$ has $\Ricci<0$, then the isometric pair $\left(\sg_{hDh^{-1}},\left\la\overline{h}^{-1}\cdot,\overline{h}^{-1}\cdot\right\ra\right)$ also has $\Ricci<0$, where $c=1$ and $X=0$ (see \eqref{hbarra}), and $hDh^{-1}$ is symmetric with respect to the metric $\left\la \overline{h}^{-1}\cdot,\overline{h}^{-1}\cdot\right\ra$.   

\begin{question}\label{Dsrn}
What kind of set is the cone $\Der(\ngo)_{srn}$? Is $\Der(\ngo)_{srn}$ open in $\Der(\ngo)^\RR$? 
\end{question}

The following result characterizes strongly Ricci negative derivations in terms of the moment map $\mm:\Lambda^2\ngo^*\otimes\ngo\rightarrow\sym(\ngo)$ for the $\Gl(\ngo)$-representation $\Lambda^2\ngo^*\otimes\ngo$, defined by 
\begin{equation}\label{defmm-2}
\la \mm(\mu),E\ra=\tfrac{1}{|\mu|^2}\left\langle E\cdot\mu,\mu\right\rangle, \qquad \mu\in
\Lambda^2\ngo^*\otimes\ngo, \quad E\in\sym(\ngo).  
\end{equation}
See Section \ref{mm-sec} for more details on the moment map.  We fix a basis $\{ e_i\}$ of $\ngo$ such that 
$$
\tg(\ngo)\subset\Diag(\ngo), 
$$ 
where $\Diag(\ngo)$ denotes the vector space of all operators of $\ngo$ whose matrix in terms of $\{ e_i\}$ is diagonal.  Note that the vector space $\ngo$ is identified with $\RR^n$ using the basis $\{ e_i\}$ and so the whole setting described in Section \ref{mm-sec} can be used.  In particular, the inner products appearing in \eqref{defmm-2} are the canonical ones, i.e.\  $\{ e_i\otimes e^j\}$ and $\{ (e^i\wedge e^j)\otimes e_k\}$ are respectively orthonormal bases.  

Let $G_D$ denote the connected component of the centralizer subgroup of $D$ in $\Gl(\ngo)$ and let $\ggo_D$ be its Lie algebra.  

\begin{theorem}\label{main}\cite[Corollary 3.4]{RN}
Let $\ngo$ be a nilpotent Lie algebra with Lie bracket $\lb$ and consider $D\in\tg(\ngo)$ such that $\tr{D}>0$.  Then the following conditions are equivalent:
\begin{itemize}
\item[(i)] $D$ is strongly Ricci negative.
\item[ ]
\item[(ii)] $D\in\RR_{>0}\, \mm\left(G_D\cdot\lb\right)\cap\Diag(\ngo) + \Diag(\ngo)_{>0}$.
\item[ ]
\item[(iii)] $D\in\RR_{>0}\, \mm\left(\overline{G_D\cdot\lb}\right)\cap\Diag(\ngo) + \Diag(\ngo)_{>0}$.
\item[ ]
\item[(iv)] $D\in\RR_{>0}\, \mm\left(\overline{G_D\cdot\lb}\right)\cap\ag_+^D + \Diag(\ngo)_{>0}$, where $\ag_+^D\subset\Diag(\ngo)$ is any Weyl chamber of $G_D$.
\end{itemize}
\end{theorem}

It would be very helpful to know the answer to the following question.  

\begin{question}\label{trDp}
Does condition (iii) imply that $\tr{D}>0$? 
\end{question}

Note that all the cones in the above theorem depend on $D$ and are open in $\Diag(\ngo)$ since the space $\Diag(\ngo)_{>0}$ of positive diagonal matrices is so.  On the other hand, the cone in part (iv) is in addition convex by Theorem \ref{conv1}.

Since we do not know if $\Der(\ngo)_{srn}$ is open in $\Der(\ngo)^\RR$ (see Question \ref{Dsrn}), we do not know whether the corresponding cone 
\begin{equation}\label{tsrn-def}
\tg(\ngo)_{srn}:=\Der(\ngo)_{srn}\cap\tg(\ngo)
\end{equation}
is open in $\tg(\ngo)$ either.  Note that the Weyl group $W(\ngo)$ (see \eqref{Wg}) is also acting on $\tg(\ngo)_{srn}$ since $\Der(\ngo)_{srn}$ is $\Aut(\ngo)$-invariant.      

The cone $\tg(\ngo)_{srn}$ was proved to be open in $\tg(\ngo)$ and convex for Heisenberg and filiform Lie algebras in \cite{NklNkn,Nkl} (see \cite{Gtr}).  Furthermore, it is shown in \cite{Gtr} that the same holds for every $\ngo$ of dimension $\leq 5$.  All this and the results given in the next section suggest that the following may be true.   

\begin{conjecture}\label{tsrn}
The cone $\tg(\ngo)_{srn}$ is open in $\tg(\ngo)$ and convex.     
\end{conjecture}

\section{An open and convex cone of Ricci negative derivations}\label{cone-sec}

In this section, we use Theorem \ref{main} to construct, for any nilpotent Lie algebra $\ngo$, an open and convex cone $\cca(\ngo)$ of diagonalizable Ricci negative derivations, containing all the generic ones.    

We keep the basis $\{ e_1,\dots,e_n\}$ of $\ngo$ such that $\tg(\ngo)\subset\Diag(\ngo)$ fixed.  Note that $\tg(\ngo)$ is precisely the orthogonal complement in $\Diag(\ngo)$ of the set $\{ F_{ij}^k:c_{ij}^k\ne 0\}$, where $c_{ij}^k$ are the structural constants of the Lie bracket $\lb$ of $\ngo$ and $F_{ij}^k$ is the diagonal matrix with nonzero entries $-1$ at $i,j$ and $1$ at $k$.  Let $\alpha_1,\dots,\alpha_r\in\tg(\ngo)^*$ be the weights for $\tg(\ngo)$ and let $\ngo=\ngo_1\oplus\dots\oplus\ngo_r$ be the decomposition of $\ngo$ in weight subspaces, that is,
$$
DX=\alpha_i(D)X, \qquad\forall X\in\ngo_i, \quad D\in\tg(\ngo).
$$
We call a derivation $D\in\tg(\ngo)$ {\it generic} when $\alpha_i(D)\ne\alpha_j(D)$ for any $i\ne j$.  Let $\tg(\ngo)_{gen}$ denote the subset of all generic derivations, which is open and dense in $\tg(\ngo)$.  

If $G_{\tg(\ngo)}$ denotes the connected component of the centralizer of $\tg(\ngo)$ in $\Gl(\ngo)$, then $G_{\tg(\ngo)}\subset G_D$ for any $D\in\tg(\ngo)$ and equality holds if and only if $D\in\tg(\ngo)_{gen}$.  In the {\it multiplicity-free} case, i.e.\ $\dim{\ngo_i}=1$ for all $i$, $G_{\tg(\ngo)}$ is given by the connected torus $\Diag(\ngo)_{>0}$ with Lie algebra $\Diag(\ngo)$ and $\tg(\ngo)$ is the only Weyl chamber.  Note that the space $\tg(\ngo)$ is multiplicity-free if and only if there is at least one $D\in\tg(\ngo)$ with pairwise different $n$ eigenvalues; in that case, the basis is necessarily nice (see Section \ref{nice-sec}).

We consider the smallest of the open cones appearing in Theorem \ref{main}, (iii) to define
\begin{equation}\label{cone-def}
\Cone(\ngo):=\left(\RR_{>0}\, m\left(\overline{G_{\tg(\ngo)}\cdot\lb}\right)\cap\Diag(\ngo) + \Diag(\ngo)_{>0}\right)\cap\tg(\ngo)_{\tr>0},
\end{equation}
where $\tg(\ngo)_{\tr>0}:=\left\{ D\in\tg(\ngo):\tr{D}>0\right\}$.  A positive answer to Question \ref{trDp} would allow us to take the intersection with $\tg(\ngo)_{\tr>0}$ out of the definition of the open cone $\cca(\ngo)$.  By Theorem \ref{main}, (iv) one can use any Weyl chamber $\ag_+\subset\Diag(\ngo)$ of $G_{\tg(\ngo)}$ instead of the whole $\Diag(\ngo)$ to define $\Cone(\ngo)$, that is, 
\begin{equation}\label{cone-def2}
\Cone(\ngo) = \left(\RR_{>0}\, m\left(\overline{G_{\tg(\ngo)}\cdot\lb}\right)\cap\ag_+ + \Diag(\ngo)_{>0}\right)\cap\tg(\ngo)_{\tr>0}. 
\end{equation}

Let $\Or(\ngo)$ denote the orthogonal group relative to the inner product making $\{ e_i\}$ an orthonormal basis.  In order to have a group acting on $\cca(\ngo)$, which may be helpful for computing it, we consider the subgroup of the Weyl group $W(\ngo)$ (see \eqref{Wg}) given by 
\begin{equation}\label{Wgort}
W_{ort}(\ngo):=N_{\Aut(\ngo)\cap\Or(\ngo)}(\tg(\ngo))/C_{\Aut(\ngo)\cap\Or(\ngo)}(\tg(\ngo))
\end{equation}
and call it the {\it orthogonal Weyl group}.   

\begin{proposition}
$\Cone(\ngo)$ is an open and convex cone in $\tg(\ngo)$ such that 
\begin{equation}\label{cone-cont}
\tg(\ngo)_{srn}\cap\tg(\ngo)_{gen}\subset\Cone(\ngo)\subset\tg(\ngo)_{srn}.
\end{equation}
Moreover, $W_{ort}(\ngo)$ acts on $\Cone(\ngo)$ by conjugation and $\Cone_t(\ngo):=\{ D\in \Cone(\ngo):\tr{D}=t\}$ is a $W_{ort}(\ngo)$-invariant convex polytope (i.e.\ the convex hull of finitely many points) for any $t>0$.  
\end{proposition}

\begin{proof}
The open cone $\Cone(\ngo)$ is convex and $\Cone_t(\ngo)$ is a polytope by Theorem \ref{conv1} and \eqref{cone-def2}.  Condition \eqref{cone-cont} follows from Theorem \ref{main} and the fact that $G_{\tg(\ngo)} = G_D$ for any  $D\in\tg(\ngo)_{gen}$.  

It only remains to prove that $\Cone(\ngo)$ is $W_{ort}(\ngo)$-invariant.  If $f\in W_{ort}(\ngo)$ and  $D\in\Cone(\ngo)$, say $D=r\mm(h\cdot\lb)+E$, for some $r>0$, $h\in G_{\tg(\ngo)}$ and $E\in\Diag(\ngo)_{>0}$, then 
$$
fDf^{-1} = rf\mm(h\cdot\lb)f^{-1}+fEf^{-1} = r\mm((fhf^{-1})\cdot\lb)+fEf^{-1},
$$   
with $fhf^{-1}\in G_{\tg(\ngo)}$ and thus $fDf^{-1}\in\cca(\ngo)$, concluding the proof.  
\end{proof}

In view of the suggestive property  \eqref{cone-cont} satisfied by $\Cone(\ngo)$, it is natural to expect that the following must hold.  

\begin{conjecture}\label{Ctsrn}
$\Cone(\ngo)=\tg(\ngo)_{srn}$.     
\end{conjecture}

This has recently been proved for any nilpotent Lie algebra of dimension $\leq 5$ in \cite{Gtr}.  Note that the conjecture holds in the rank one case (i.e. $\dim{\tg(\ngo)}=1$).  Furthermore, in \cite{Gtr}, the cone $\cca(\ngo)$ is computed for some classical families of nilpotent Lie algebras, obtaining for example that the convex polytope $\cca_1(\hg_{2n+1})$ is the $n$-dimensional hypercube for the Heisenberg Lie algebra $\hg_{2n+1}$.  

The following result shows that we only need $\tg(\ngo)_{srn}$ to be open in $\tg(\ngo)$ to prove Conjectures \ref{tsrn} and \ref{Ctsrn}.  

\begin{proposition}
If $\tg(\ngo)_{srn}$ is open in $\tg(\ngo)$, then $\Cone(\ngo)=\tg(\ngo)_{srn}$.
\end{proposition}

\begin{proof}
For a given $D\in\tg(\ngo)_{srn}$, consider an open neighborhood $U\subset \tg(\ngo)_{srn}$.  It is easy to show that $D$ belongs to a segment contained in $U$ with extremes in $\tg(\ngo)_{gen}$, and since such extremes belong to the convex cone $\Cone(\ngo)$ by \eqref{cone-cont} we obtain that $D\in\Cone(\ngo)$, concluding the proof.
\end{proof}

Let us summarize the different cones in the maximal torus $\tg(\ngo)\subset\Der(\ngo)$ that we have considered in this and the previous sections, we refer to \eqref{trn-def}, Definition \ref{rnd-def}, \eqref{tsrn-def} and \eqref{cone-def} for their definition, respectively:
\begin{equation}\label{conts}
\tg(\ngo) \supseteq \tg(\ngo)_{rn} \supseteq \Der(\ngo)_{rn}\cap\tg(\ngo) \supseteq \tg(\ngo)_{srn} \supseteq \cca(\ngo). 
\end{equation}
They are all known to be open except $\tg(\ngo)_{srn}$; on the other hand, $\cca(\ngo)$ is the only one known to be convex (see Conjectures \ref{trn}, \ref{Diff} and \ref{Ctsrn}).  

\begin{remark}
It is important to notice that throughout all this section and Section \ref{srn-sec}, $\tg(\ngo)$ can be replaced by $\Der(\ngo)\cap\Diag(\ngo)$, where $\Diag(\ngo)$ is defined with respect to any basis of $\ngo$.  It is not necessary for $\Der(\ngo)\cap\Diag(\ngo)$ to be maximal.  
\end{remark}

\subsection{Nice bases}\label{nice-sec}
The following type of bases is very friendly with the computation of Ricci curvature.  

\begin{definition}\label{nice-def}
A basis $\{ e_1,\dots,e_n\}$ of a Lie algebra is said to be {\it nice} if every bracket $[e_i,e_j]$ is a scalar multiple of some element $e_k$ in the basis and two different brackets $[e_i,e_j]$, $[e_r,e_s]$ can be a nonzero multiple of the same $e_k$ only if $\{ i,j\}$ and $\{ r,s\}$ are either equal or disjoint.
\end{definition}

It is proved in \cite{nicebasis} (see also \cite[Lemma 3.12]{RN}) that a basis $\{ e_i\}$ is a nice basis if and only if
\begin{equation}\label{nice-m}
\mm\left(\overline{\Diag(\ngo)_{>0}\cdot\lb}\right) = \CH_{\lb}:= \CH(F_{ij}^k:c_{ij}^k\ne 0),
\end{equation}
where $c_{ij}^k$ are the structure coefficients and $\CH$ denotes convex hull.  If  $\{ e_i\}$ is a nice basis, then by \eqref{nice-m},
$$
\CH_{\lb}\subset \mm\left(\overline{G_{\tg(\ngo)}\cdot\lb}\right)\cap\Diag(\ngo),
$$
and equality holds if in addition $\tg(\ngo)$ is multiplicity-free since in that case, $G_{\tg(\ngo)}=\Diag(\ngo)_{>0}$ and $\ag_+=\Diag(\ngo)$ is the only Weyl chamber.

The following example shows that both $\mm\left(\overline{\Diag(\ngo)_{>0}\cdot\lb}\right)\cap\Diag(\ngo)$ and $\mm\left(\overline{G_{\tg(\ngo)}\cdot\lb}\right)\cap\Diag(\ngo)$ can be tricky subsets if the basis is not nice.  

\begin{example}\label{tricky} (see \cite[Example 3.17]{RN} for a detailed treatment). 
Let $\ngo$ be the $5$-dimensional nilpotent Lie algebra with basis $\{ e_1,\dots,e_5\}$ and Lie bracket
$$
[e_1,e_2]=e_3+e_4, \quad [e_1,e_3]=e_5, \quad [e_1,e_4]=e_5,
$$
for which $\CH_{\lb}=\CH(F_{12}^3,F_{12}^4,F_{13}^5,F_{14}^5)$ is a rectangle centered at $\mm(\lb)$ (see Figure \ref{tricky-fig}) and  
$$
G_{\tg(\ngo)}=\left\{\left[\begin{smallmatrix}
h_1&&&\\
&h_2&&\\
&&H&\\
&&&h_5
\end{smallmatrix}\right]: H\in\Gl_2^+(\RR), \quad h_i>0\right\}. 
$$
We consider $G_{\tg(\ngo)}$ acting on the cone $\overline{G_{\tg(\ngo)}\cdot\lb}$ of nilpotent Lie brackets $\mu=\mu(x,y,z,w)$ defined by
$$
\mu(e_1,e_2)=xe_3+ye_4, \quad \mu(e_1,e_3)=ze_5, \quad \mu(e_1,e_4)=we_5, \qquad x,y,z,w\geq 0.
$$
The moment map $\mm:\overline{G_{\tg(\ngo)}\cdot\lb}\longrightarrow\sym(\ngo)$ is given by
$$
\mm(\mu) = \tfrac{1}{x^2+y^2+z^2+w^2}\left(x^2F_{12}^3+y^2F_{12}^4+z^2F_{13}^5+w^2F_{14}^5+(xy-zw)(E_{34}+E_{43})\right),  
$$
where $E_{ij}$ denotes as usual the matrix whose only nonzero entry is a $1$ at $ij$, and so
\begin{align}
\mm\left(\overline{G_{\tg(\ngo)}\cdot\lb}\right)\cap\Diag(\ngo) =& \{ aF_{12}^3+bF_{12}^4+cF_{13}^5+dF_{14}^5: a,b,c,d\geq 0, \label{image1}\\
&\quad a+b+c+d=1, \quad ab=cd\}. \notag
\end{align}
It is not so hard to see by using \eqref{image1} that $\mm\left(\overline{G_{\tg(\ngo)}\cdot\lb}\right)\cap\Diag(\ngo)$ is the union of two triangles as in Figure \ref{tricky-fig}, each of which being the intersection with one of the Weyl chambers 
$$
(\ag_+)_1=\{\Diag(a_1,\dots,a_5):a_3\leq a_4\}, \qquad (\ag_+)_2=\{\Diag(a_1,\dots,a_5):a_3\geq a_4\}. 
$$
On the other hand,  the closure $\overline{\Diag(\ngo)_{>0}\cdot\lb}$ is the union of the open orbit $\Diag(\ngo)_{>0}\cdot\lb=\{\mu:xz=yw\ne 0\}$ and other eight nonzero $T$-orbits corresponding to the edges and vertexes of the rectangle $\CH_{\lb}$ (see \eqref{deg-caras}).   One therefore obtains from \eqref{image1} that
\begin{align}
\mm\left(\overline{\Diag(\ngo)_{>0}\cdot\lb}\right)\cap\Diag(\ngo) =& \{ aF_{12}^3+bF_{12}^4+cF_{13}^5+dF_{14}^5: a,b,c,d\geq 0, \label{image3}\\
&\quad a+b+c+d=1, \quad ab=cd, \quad ac=bd\},  \notag
\end{align}
given by the union of three segments as in Figure \ref{tricky-fig}.  
\end{example}

\begin{figure}
\begin{tikzpicture}[scale=0.7]
\draw[very thin,color=gray,step=1cm] (-1.5,-2.5) grid (7.5,6.5);

\draw[fill=orange, opacity=0.8] (0,0) -- (6,0) -- (3,2.1) -- (0,0);
\draw[fill=orange, opacity=0.8]  (3,2.1) -- (6,4.2) -- (0,4.2) -- (3,2.1);
\draw  (0,0) -- (0,4.2);
\draw  (6,0) -- (6,4.2);
\draw[ultra thick, blue]  (0,0) -- (6,0);
\draw[ultra thick, blue]  (3,0) -- (3,4.2);
\draw[ultra thick, blue]  (0,4.2) -- (6,4.2);

\draw[fill] (3,2.1) circle (2pt) node[right] {$\mm(\lb)$}; 
\draw (3,5.5) node[thick, orange] {$\mm\left(\overline{G_{\tg(\ngo)}\cdot\lb}\right)\cap \Diag(\ngo)$};
\draw (3,-1.5) node[thick, blue] {$\mm\left(\overline{\Diag(\ngo)_{>0}\cdot\lb}\right)\cap\Diag(\ngo)$};

\draw[fill] (0,0) circle (2pt) node[below left] {$F_{14}^5$}; 
\draw[fill] (6,0) circle (2pt) node[below right] {$F_{12}^3$}; 
\draw[fill] (0,4.2) circle (2pt) node[above left] {$F_{13}^5$}; 
\draw[fill] (6,4.2) circle (2pt) node[above right] {$F_{12}^4$}; 
\end{tikzpicture}
\caption{Moment map images}\label{tricky-fig}
\end{figure}
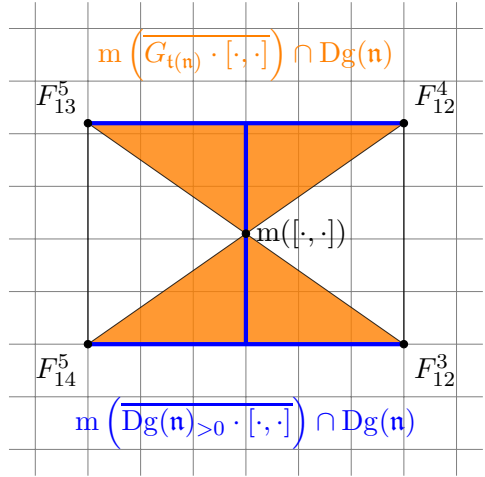

The following fact from Lie theory will be useful.  Recall that any ideal of a nilpotent Lie algebra meets the center non-trivially by Engel's Theorem.   

\begin{lemma}\label{Lie}
Let $\ngo$ be a nilpotent Lie algebra and assume that $\ngo=\ngo_0\oplus\ngo_1$, where $\ngo_0$ is a subalgebra, $\ngo_1$ an ideal and $\zg(\ngo)\subset\ngo_1$.  Then for any nonzero $X_0\in\ngo_0$, there exists $X_1\in\ngo_1$ such that $[X_0,X_1]\ne 0$.
\end{lemma}

\begin{proof}
Assume there is a nonzero $X_0\in\ngo_0$ such that $[X_0,\ngo_1]=0$.  Since $\ngo_0$ is nilpotent and $X_0\notin\zg(\ngo)$, there exists $k\in\NN$ such that
$$
\bg_k:= [\dots[X_0,\ngo_0],\dots,\ngo_0] \quad (k\mbox{-times})
$$
is nonzero and $\bg_{k+1}=0$.  It follows by induction that $[\bg_k,\ngo_1]=0$ for all $k$; indeed, for any  $Z\in\bg_k$, $X\in\ngo_0$ (i.e. $[Z,X]\in\bg_{k+1}$) and  $Y\in\ngo_1$,
$$
[[Z,X],Y] = [[Z,Y],X] + [Z,[X,Y]] = 0, 
$$
by inductive hypothesis.  Thus $\bg_k\subset\ngo_0\cap\zg(\ngo)=\{ 0\}$, which is a contradiction.
\end{proof}

As an application, we next show that any nilpotent Lie algebra admitting a non-negative diagonalizable (relative to a nice basis) derivation which is positive on the center is a RN-nilradical.  

\begin{proposition}
Let $\ngo$ be a nilpotent Lie algebra endowed with a nice basis $\{ e_i\}$.  If $D\in\Der(\ngo)\cap\Diag(\ngo)$ satisfies that $D\geq 0$ and $D|_{\zg(\ngo)}>0$, then $D\in\Der(\ngo)_{srn}$.
\end{proposition}

\begin{proof}
We consider the decomposition $\ngo=\ngo_0\oplus\ngo_1$, where $\ngo_0$ is the kernel of $D$ and $\ngo_1$ the sum of all eigenspaces of $D$ with positive eigenvalues.  Note that such decomposition satisfies the hypothesis of Lemma \ref{Lie}.  We can assume that $\{ e_1,\dots,e_r\}$ is a basis for $\ngo_0$.  It follows from Lemma \ref{Lie} that for each $i=1,\dots,r$, there exist $j_i,k_i$ such that $e_{j_i},e_{k_i}\in\ngo_1$ and $c_{ij_i}^{k_i}\ne 0$.  Thus
$$
M:=F_{1j_1}^{k_1}+\dots+F_{rj_r}^{k_r},
$$
satisfies that $M|_{\ngo_0}<0$ and so $D-\epsilon M>0$ for some $\epsilon >0$ since $D|_{\ngo_1}>0$.  Since $M\in \RR_{>0}m\left(\overline{\Diag(\ngo)_{>0}\cdot\lb}\right)$ by \eqref{nice-m}, we conclude from Theorem \ref{main}, (iii) that $D\in\Der(\ngo)_{srn}$.
\end{proof}

\section{Concluding remarks} 

The present state of the art of the main question of which Lie algebras admit $\Ricci<0$ metrics can be summarized as follows: 

\begin{enumerate}[1)]
\item[ ]
\item {\it Unimodular and solvable}: not allowed.
\item[ ]
\item {\it Unimodular and non-solvable}: only semisimple Lie algebras are allowed and most non-compact simple Lie algebras are known to admit.  
\item[ ]
\item {\it Non-unimodular and non-solvable}: there are several large classes of examples, including all compact and most non-compact semisimple Lie algebras as possible Levi factors as well as many of their representations as possible nilradicals.  On the other hand, no obstruction is known and not a single example where the non-existence of $\Ricci<0$ metrics can be proved is available.    
\item[ ]
\item {\it Non-unimodular and solvable}: only one obstruction is available and if all the conjectures in this paper turn out to be true, then $\sg$ admits $\Ricci<0$ metrics if and only if the diagonalizable part of some $\ad{Y}$ restricted to the nilradical $\ngo$ belongs (up to automorphism conjugation) to the open and convex cone of derivations $\cca(\ngo)\subset\tg(\ngo)$ introduced in Section \ref{cone-sec}.  
\item[ ]
\end{enumerate}

Independently of the truth of the characterization given in 4), it is important to highlight the role of the beautiful open and convex cone $\cca(\ngo)$ defined via a moment map, in actually producing $\Ricci<0$ metrics on different solvable Lie algebras with the same nilpotent Lie algebra $\ngo$ as nilradical.

\section{Appendix: the moment map}\label{mm-sec}

We consider the vector space $V:=\lam$ of all skew-symmetric algebras of dimension $n$, on which $\G$ is acting by $g\cdot\mu:=g\mu(g^{-1}\cdot,g^{-1}\cdot)$ and $\g$ by
$$
E\cdot\mu=E\mu(\cdot,\cdot)-\mu(E\cdot,\cdot)-\mu(\cdot,E\cdot),\qquad E\in\g,\quad\mu\in V.
$$
The weight vectors for the above representation are $\mu_{ijk}:=(e^i\wedge e^j)\otimes e_k$ with corresponding weights
$$
F_{ij}^k:=E_{kk}-E_{ii}-E_{jj}\in\tg^n, \qquad i<j,
$$
where $\tg^n$ denotes the set of all diagonal $n\times n$ matrices.  The structural constants $c(\mu)_{ij}^k$ of an
algebra $\mu\in V$ are therefore given by
$$
\mu(e_i,e_j)=\sum_{k}c(\mu)_{ij}^k\, e_k,
\qquad \mu=\sum_{i<j,\,k} c(\mu)_{ij}^k\, \mu_{ijk}.
$$
All these vector spaces will be endowed with their canonical inner products.

The moment map (or $\G$-gradient map) from real geometric invariant theory (see \ \cite{HnzSchStt,BhmLfn1} for further information) for the above representation is the $\Or(n)$-equivariant map
$$
\mm:V\smallsetminus\{ 0\}\longrightarrow\sym(n),
$$
defined by
\begin{equation}\label{defmm}
\la \mm(\mu),E\ra=\tfrac{1}{|\mu|^2}\left\langle E\cdot\mu,\mu\right\rangle, \qquad \mu\in
V\smallsetminus\{ 0\}, \quad E\in\sym(n), 
\end{equation}
or equivalently, for any $X,Y\in\RR^n$, 
\begin{equation}\label{defmm2}
\la \mm(\mu)X,Y\ra = -\unm\sum\la\mu(X,e_i),e_j\ra\la\mu(Y,e_i),e_j\ra + \unc \sum\la\mu(e_i,e_j),X\ra\la\mu(e_i,e_j),Y\ra.   
\end{equation}
We are using $\g=\sog(n)\oplus\sym(n)$ as a Cartan decomposition, where $\sog(n)$ and $\sym(n)$ denote the subspaces of skew-symmetric and symmetric matrices, respectively.  Note that $\mm$ is also well defined on the projective space $\PP(V)$ and that $\tr{\mm(\mu)}=-1$ for any $\mu\in V$.  In \cite{HnzSch, BllGhgHnz} many nice and useful results on the convexity of the image of the moment map have been obtained, which were used in \cite{RN} to study Ricci negative solvmanifolds (see Section \ref{RN-sec}).

For any compatible subgroup $G\subset\Gl_n(\RR)$, one has that $K:=G\cap\Or(n)$ is a maximal compact subgroup of $G$, $\ggo=\kg\oplus\pg$ is a Cartan decomposition, where $\pg:=\ggo\cap\sym(n)$, and $G=K\exp{\pg}$.  Consider $\ag\subset\pg$, a maximal abelian subalgebra; thus the corresponding torus $A=\exp{\ag}\subset G$ is also a compatible subgroup.  Let $\mm$ and $\mm_\ag$ denote the moment maps for the $G$-action and for the $A$-action, respectively, and  let $\ag_+\subset\ag$ be any Weyl chamber of $G$.  Consider also $W:=N_K(\ag)/C_K(\ag)$, the corresponding Weyl group.

\begin{theorem}\label{conv1}\cite{HnzSch}
For any closed $G$-invariant subset $X\subset\PP(V)$, $\mm(X)\cap\ag$ is the union of finitely many convex polytopes.  If in addition $X$ is irreducible, then $\mm(X)\cap\ag_+$ is a convex polytope and
$$
\mm(X)\cap\ag = \bigcup\limits_{k\in W} k\cdot(\mm(X)\cap\ag_+).
$$
\end{theorem}

\begin{remark}
In particular, $\mm_\ag(X)$ is always a convex polytope if $X$ is irreducible since $\ag$ is the only Weyl chamber of $A$.  
\end{remark}

Therefore $\mm(\overline{G\cdot\mu})\cap\ag_+$ and $\mm_\ag(\overline{G\cdot\mu})$ are both convex polytopes for any $\mu\in V$ if $G$ is connected, but $\mm(\overline{G\cdot\mu})\cap\ag$ and $\mm(\overline{A\cdot\mu})\cap\ag$ are not convex in general (see Example \ref{tricky}).

For each $\mu\in V$, we define the following convex polytope in $\Diag(\ngo)$,
$$
\CH_\mu := \CH(F_{ij}^k:c(\mu)_{ij}^k\ne 0).
$$
Note that $F_{ij}^k=\mm(\mu_{ijk})$.  It is proved in \cite[Section 3.2]{RN} that
\begin{equation}\label{mtoro}
\Diagg\left(\mm\left(\overline{\Diag(\ngo)_{>0}\cdot\mu}\right)\right) = \CH_\mu, \qquad \Diagg\left(\mm\left(\Diag(\ngo)_{>0}\cdot\mu\right)\right) = \CH_\mu^\circ,
\end{equation}
where $\Diagg(A)$ denotes the diagonal part of a matrix $A$.  Conversely, $\overline{\Diag(\ngo)_{>0}\cdot\mu}$ is combinatorically determined by $\CH_\mu$ in the following way (see \cite[Lemma 3.10]{RN}):
\begin{equation}\label{deg-caras}
\overline{\Diag(\ngo)_{>0}\cdot\mu} =\left\{\lambda_J:\CH(F_{ij}^k:(i,j,k)\in J)\;\mbox{is a face of}\; \CH_\mu\right\},
\end{equation}
where for each subset $J\subset I_\mu:=\{ (i,j,k):c(\mu)_{ij}^k\ne 0\}$,
$$
\lambda_J:=\sum_{(i,j,k)\in J} c(\mu)_{ij}^k\,\mu_{ijk}.
$$
Moreover, if $c(\mu)_{ij}^k\ne 0$, then $\mu_{ij}^k\in\overline{\Diag(\ngo)_{>0}\cdot\mu}$ and $F_{ij}^k$ is an extreme point of $\CH_\mu$.  In particular,
\begin{equation}\label{FenT}
  F_{ij}^k\in \mm(\overline{\Diag(\ngo)_{>0}\cdot\mu})\cap\tg^n, \qquad \forall\; c(\mu)_{ij}^k\ne 0.
\end{equation}

\end{document}